\documentclass[10pt]{amsart} 
\def\update{November 24,  2013}  

\usepackage{amssymb} 

\newtheorem{theorem}{Theorem}[]
\newtheorem{lemma}{Lemma}
\newtheorem{corollary}{Corollary}
\newtheorem{proposition}{Proposition} 
\newtheorem*{proposition*}{Proposition} 

\def\trd{{\mathrm{trdeg}}}
\def\mathbbL{{\mathbb L}}
\def\mathbbLc{{\mathbb L}^{ \! \mathrm c}}
\def\C{{\mathbf C}}
\def\Q{{\mathbf Q}}
\def\R{{\mathbf R}}
\def\X{{\mathbf X}}

\def\Z{{\mathbf Z}}
\def\Qbar{\overline{\Q}}

\def\calE{{\mathcal{E}}}
\def\calI{{\mathcal{I}}}
\def\calJ{{\mathcal{J}}}
\def\calV{{\mathcal{V}}} 
\def\calN{{\mathcal{N}}} 
 
\begin{document}

\title[Liouville Numbers and Algebraic independence]{ Liouville Numbers and Schanuel's Conjecture
}
\author{K. Senthil Kumar, R. Thangadurai and M. Waldschmidt} 
\address[K. Senthil Kumar and R. Thangadurai]{Harish-Chandra Research Institute, Chhatnag Road, Jhunsi, Allahbad, 211019, India}
\address[M. Waldschmidt]{Institut de Math\'ematiques de Jussieu, 
Th\'eorie des Nombres Case courrier 247, 
Universit\'e Pierre et Marie Curie (Paris 6), 
PARIS Cedex 05 FRANCE}
\email[K. Senthil Kumar]{senthil@hri.res.in}
\email[R. Thangadurai]{thanga@hri.res.in}
\email[M. Waldschmidt]{miw@math.jussieu.fr}

\begin{abstract} 
In this paper, using an argument of P.~Erd\H{o}s, K.~Alnia\c{c}ik and \'{E}.~Saias, we extend earlier results on Liouville numbers, due to P.~Erd\H{o}s, G.J.~Rieger, W.~Schwarz, K.~Alnia\c{c}ik, \'{E}.~Saias, 
E.B.~Burger. We also produce new results of algebraic independence related with Liouville numbers and Schanuel's Conjecture, in the framework of $G_\delta$--subsets. 

\null
\hfill {\bf Update: \update }
\end{abstract}
\maketitle

\section{ Introduction}\label{Section:Introduction}

For any integer $q$ and any real number $x \in \R$, we denote by
$$
\Vert qx\Vert = \min_{m\in{ \Z}} |qx -m|
$$
the distance of $qx$ to the nearest integer. Following \'E.~Maillet \cite{Maillet}, an irrational real number $\xi$ is said to be a {\it Liouville number} if, for each integer $n\geq 1$, there exists an integer $q_n \geq 2$ such that the sequence $\bigl(u_n(\xi)\bigr)_{n\ge 1}$ of real numbers defined by 
$$
u_n(\xi) = -\frac{\log\Vert q_n\xi \Vert}{\log q_n}
$$ 
satisfies 
$\displaystyle \lim_{n\to\infty} u_n(\xi) = \infty$.
If $p_n$ is the integer such that $\Vert q_n\xi\Vert = |q_n\xi  - p_n|$, then the definition of $u_n(\xi)$ can be written
$$
\left| q_n \xi - p_n \right| = \frac{1}{q_n^{u_n(\xi)}}\cdotp
$$
An equivalent definition is to say that a Liouville number is a real number $\xi$ such that, for each integer $n\geq 1$, there exists a rational number $p_n/q_n$ with $q_n \geq 2$ such that
$$
0<\left| \xi-\frac{p_n}{q_n}\right| \le \frac{1}{q_n^n}\cdotp
$$
We denote by ${\mathbbL}$ the set of Liouville numbers. This set ${\mathbbL}$ is an uncountable, dense subset of ${\R}$ having Lebesgue measure $0$ and  
$$
\mathbbL=\bigcap_{n\ge 1}U_n
\quad
\hbox{with}
\quad 
U_n=\bigcup_{q\ge 2} \bigcup _{p\in\Z} \left(
\frac{p}{q}-\frac{1}{q^n}
\; , \; 
\frac{p}{q}+\frac{1}{q^n}
\right)\setminus
\left\{\frac{p}{q}
\right\}
\cdotp
$$
Each $U_n$ is dense, since each $p/q\in\Q$ belongs to the closure of $U_n$. Throughout this article, a $G_\delta$--subset of a topological space $X$ is defined to be the countable intersection of dense open subsets of $X$.  Baire's theorem states that in a complete or locally compact space $X$, any $G_\delta$-subset is dense. Since $\R$ is complete, we see that $\mathbbL$ is a $G_\delta$--subset of $\R$.  {\tt The} $G_\delta$--subset is also defined as a set having a complement which is {\it meager}. In our case, this complement $\mathbbLc$ is the set of non--Liouville numbers
\begin{align}\notag
\mathbbLc=\Bigl\{
x\in\R\; \mid\;
\; 
\hbox{there exists $\kappa>0$ such that }& \; 
\\
\notag 
\left|x-\frac{p}{q}\right|
&
\ge \frac{1}{q^\kappa}
\; 
\hbox{for all } \; 
\frac{p}{q}\in\Q \; \hbox{ with }\; q\ge 2
\Bigr\},
\end{align}
which has full Lebesgue measure.

In 1844, J. Liouville \cite{liouville1} proved that any element of ${\mathbbL}$ is a transcendental number.
A survey on algebraic independence results related with Liouville numbers is given in \cite{wald2}.

In 1962, P.~Erd\H{o}s \cite{erd1} proved that {\it every real number $t$ can be written as $t= \xi + \eta$ with $\xi$ and $\eta$ Liouville numbers}. He gave two proofs of this result. The first one is elementary and constructive: he splits the binary expansion of $t$ into two parts, giving rise to binary expansions of two real numbers $\xi$ and $\eta$, the sum of which is $t$. The splitting is done in such a way that both binary expansions of $\xi$ and $\eta$ have long sequences of $0$'s. The second proof is not constructive, as it relies on Baire's Theorem. In the same paper, P.~Erd\H{o}s gives also in the same way two proofs, a constructive one and another depending on Baire's Theorem, that {\it every non-zero real number $t$ can be written as $t= \xi\eta$, where $\xi $ and $ \eta$ are in ${\mathbbL}$}.
From each of these proofs, it follows that there exist uncountably many such representations $t=\xi+\eta$ (resp{.} $t=\xi\eta$) for a given $t$. 
Many authors extended this result in various ways: G.J.~Rieger in 1975 \cite{Rie}, W.~Schwarz in 1977 \cite{Schw}, K.~Alnia\c{c}ik in 1990 \cite{aln1}, K.~Alnia\c{c}ik and \'{E}.~Saias in 1994 \cite{AlniacikSaias}, E.B.~Burger in 1996 \cite{bur1} and 2001 \cite{bur2}. In \cite{bur1}, Burger extended Erd\H{o}s's result to a very large class of functions, including $f(x,y)=x+y$ and $g(x,y)=xy$.

Recall \cite{AlniacikSaias} that a real function $f:\calI\rightarrow\R$ is {\it nowhere locally constant}
if, for every nonempty open interval $\calJ$ contained in $\calI$, the restriction to $\calJ$ of $f$ is not constant. We define in a similar way a function which is {\it nowhere locally zero}: for every nonempty open interval $\calJ$ contained in $\calI$, the restriction to $\calJ$ of $f$ is not the zero function.

The main result of \cite{AlniacikSaias}, which extends the earlier results of \cite{Rie} and \cite{Schw}, deals with $G_\delta$--subsets, and reads as follows: 

\begin{proposition}[Alnia\c{c}ik--Saias]\label{Proposition:alnia}
{\it 
Let $\calI$ be an interval of $\R$ with nonempty interior, $G$ a $G_\delta$--subset of $\R$ and $(f_n)_{n\ge 0}$ a sequence of real maps on $\calI$, which are continuous and nowhere locally constant. Then 
$$
\bigcap_{n\ge 0} f_n^{-1}(G)
$$ 
is a $G_\delta$--subset of $\calI$. 
}
\end{proposition}

As pointed out by the authors of \cite{AlniacikSaias}, the proofs of several papers on this topic just reproduce the proof of Baire's Theorem. Here we use Baire's Theorem and deduce a number of consequences related with Liouville numbers in the subsequent sections. 

 In section $\ref{section:ApplicationLiouville}$, we deduce corollaries from Proposition $\ref{Proposition:alnia}$.
In section $\ref{section:ai}$, we deduce from Proposition $\ref{Proposition:phiandpsiRefined}$ some results of algebraic independence for Liouville numbers related to Schanuel's Conjecture.

 \bigskip
 
 \noindent
 {\it Acknowledgments.} The authors are thankful to the referee for valuable comments and suggestions. In particular he pointed out that most results in our paper are not specific to Liouville numbers: they hold with any $G_\delta$--subet of $\R$ instead of ${\mathbbL}$. So we paid attention to use the fact that ${\mathbbL}$ is a $G_\delta$--subset rather than other Diophantine properties. 
 
 \section{ Preliminaries}\label{Section:Preliminaries}
 



The following Proposition generalizes  Proposition $\ref{Proposition:alnia}$. We replace the interval $\calI$ by a topological space $\X$ and we replace $\R$ by an interval $\calJ$.

\begin{proposition}\label{Proposition:phiandpsiRefined}
Let $\X$ be a complete, locally connected topological space,  $\calJ$  an interval in $\R$ 
and
$\calN$ a set which is either finite or else countable.
For each $n\in\calN$, let $G_n$ be a $G_\delta$--subset of $\calJ$
and let  $f_n : \X\rightarrow \calJ$ be a continuous function which is nowhere locally constant. Then $\displaystyle\bigcap_{n\in \calN}f_n^{-1}(G_n)$ is a $G_\delta$--subset of $\X$. 
\end{proposition}

\begin{proof}
Since $\calN$ is at most countable, it is enough to prove for any $n\in\calN$ that 
 $f_n^{-1}(G_n)$ is a $G_\delta$--subset of $\X$. 

Since $f_n$ is continuous,%
$f_n^{-1}(G_n)$ is a countable intersection of open sets in $\X$. To prove it is  a $G_\delta$--subset of $\X$,  we need to prove that  $f_n^{-1}(G_n)$ is dense in $\X$, using the assumption that $f_n$ is nowhere locally constant.  Let $V$ be a connected open subset of $\X$. Since $f_n$ is continuous, $f_n(V)$ is a connected subset of $\calJ$. Since $f_n$ is nowhere locally constant, $f_n(V)$ consists of at least two elements.  Therefore, there exists an interval $(a, b)\subset \calJ$ with non-empty interior such that $(a, b) \subset f_n(V)$. Since $G_n$ is a dense subset of $\calJ$, we have $(a, b) \cap G_n \ne \emptyset$ and hence $V\cap f_n^{-1}(G_n) \ne \emptyset$ 
which proves Proposition   $\ref{Proposition:phiandpsiRefined}$.
\end{proof}


We close this section with the following Lemmas and a Corollary (quoted in  \cite{AlniacikSaias}). 

\begin{lemma}\label{lemma:GdeltaUncountable}
Let $\X$ be a (nonempty) complete metric 
space without isolated point 
and 
let $E$ be a $G_\delta$--subset of $\X$.  
%
Let $F$ be a countable subset of $E$. Then $E\setminus F$ is a $G_\delta$--subset of $\X$.
\end{lemma}


\begin{proof}
We have 
$$
E\setminus F=\bigcap_{y\in F} E\setminus \{y\}
$$
where each $ E\setminus \{y\}$ is a $G_\delta$--subset of $\X$ (since $\X$ has no isolated point). 
\end{proof}

Using Baire's theorem, we deduce the following Corollary 1.

\begin{corollary}\label{Corollary:uncountable}
Let $\X$ be a (nonempty) complete metric 
space without isolated point 
and 
let $E$ be a $G_\delta$--subset of $\X$.  
Then $E$ is uncountable.
\end{corollary}

The next auxiliary lemma will be useful in section $\ref{section:ApplicationLiouville}$ (proof of Corollary $\ref{Corollary:ellVariables}$). 

\begin{lemma}\label{lemma:algind}
Let $\calI_1,\dots,\calI_n$ be non-empty open subsets of $\R$. For each $i=1,\dots,n$, let $G_i$ be a $G_\delta$--subset of $\calI_i$. Then there exists $(\xi_1,\dots,\xi_n)\in G_1\times \cdots\times G_n$ such that $\xi_1,\dots,\xi_n$ are algebraically independent (over $\Q$). 
\end{lemma}

\begin{proof}
We prove Lemma $\ref{lemma:algind}$ by induction on $n$. For $n=1$, it follows from the fact that the intersection of $G_1$ with the set of transcendental numbers is not empty. 

Assume  Lemma $\ref{lemma:algind}$  holds for $n-1$. There exists $(\xi_1,\dots,\xi_{n-1})\in G_1\times \cdots\times G_{n-1}$ such that $\xi_1,\dots,\xi_{n-1}$ are algebraically independent. The set of $\xi_n\in \calI_n$ which are transcendental over $\Q(\xi_1,\dots,\xi_{n-1})$ is a $G_\delta$--subset of $\calI_n$, hence its intersection with $G_n$ is again a $G_\delta$-subset: it is dense by Baire's theorem, and therefore not empty. 

\end{proof}

\section{Application of Proposition $\ref{Proposition:alnia}$ to Liouville numbers}\label{section:ApplicationLiouville}

Since the set of Liouville numbers is a $G_\delta$--subset in $\R$, a direct consequence of Proposition $\ref{Proposition:alnia}$  and Corollary $\ref{Corollary:uncountable}$ 
 is the following:

\begin{corollary}\label{Corollary:AS}
{\it 
Let $\calI$ be an interval of $\R$ with nonempty interior and $(f_n)_{n\ge 1}$ a sequence of real maps on $\calI$, which are continuous and nowhere locally constant. Then there exists an uncountable subset $E$ of $\calI\cap\mathbbL$
such that $f_n(\xi)$ is a Liouville number for all $n\ge 1$ and all $\xi\in E$. 
}
\end{corollary}


Define $f_0:\calI\rightarrow \R$ as the identity $f_0(x)=x$. 
The conclusion of the Corollary $\ref{Corollary:AS}$ is that 
$$
E = \bigcap_{n\ge 0} f_n^{-1}\bigl(\mathbbL\bigr), 
$$
a subset of $\calI$, is  uncountable. 

In this section we deduce consequences of Corollary $\ref{Corollary:AS}$. 
We first consider the special case 
where all $f_n$ are the same. 

\begin{corollary}\label{Corollaire:unseulphi}
Let $\calI$ be an interval of $\R$ with nonempty interior and $\varphi:\calI\rightarrow\R$ a continuous map which is nowhere locally constant. Then there exists an uncountable set of Liouville numbers $\xi\in\calI$ such that $\varphi(\xi)$ is a Liouville number. 
\end{corollary}

One can deduce Corollary $\ref{Corollaire:unseulphi}$ directly  from Proposition $\ref{Proposition:alnia}$ 
by taking all $f_n=\varphi$ ($n\ge 0$) and $G=\mathbbL$ and by noticing that the intersection of the two $G_\delta$--subsets $\varphi^{-1}(\mathbbL)$ and $\mathbbL$ is uncountable. Another proof is to use  Proposition  $\ref{Proposition:alnia}$ 
 with $f_0(x)=x$ and $f_n(x)=\varphi(x)$ for $n\ge 1$ and $G=\mathbbL$. 

Simple examples of consequences of Corollary $\ref{Corollaire:unseulphi}$ are obtained with $\calI=(0,+\infty)$ and either $\varphi(x)=t-x$, for $t\in \R$, 
 or else with 
$\varphi(x)=t/x$, for $t\in \R^\times$, which yield Erd\H{o}s above mentioned result on the decomposition of any real number (resp.~any nonzero real number) $t$ as a sum (resp.~a product) of two Liouville numbers. 

We deduce also from Corollary $\ref{Corollaire:unseulphi}$ that any positive real number $t$ is the sum of two squares of Liouville numbers. This follows by applying Corollary $\ref{Corollaire:unseulphi}$ with 
$$
\calI
=(0,\sqrt{t})\quad\hbox{and}
\quad 
\varphi(x)=\sqrt{t-x^2}.
$$ 
Similar examples can be derived from Corollary $\ref{Corollaire:unseulphi}$ involving transcendental functions: for instance, any real number can be written $e^\xi+\eta$ with $\xi $ and $\eta$ Liouville numbers; any positive real number can be written $e^\xi+e^\eta$ with $\xi $ and $\eta$ Liouville numbers.

Using the implicit function theorem, one deduces from Corollary $\ref{Corollaire:unseulphi}$ the following generalization of Erd\H{o}s's result. 

\begin{corollary}\label{Corollary:refinementErdos}
Let $P\in\R[X,Y]$ be an irreducible polynomial such that $(\partial /\partial X) P \not=0$ and $(\partial /\partial Y) P \not=0$. Assume that there exist two nonempty open intervals $\calI$ and $\calJ$ of $\R$ such that, for any $x\in\calI$, there exists $y\in\calJ$ with $P(x,y)=0$, and, for any $y\in\calJ$, there exists $x\in\calI$ with $P(x,y)=0$. 
Then there exist uncountably many pairs $(\xi,\eta)$ of Liouville numbers in $\calI\times\calJ$ such that $P(\xi,\eta)=0$. 
\end{corollary}

\begin{proof}[Proof of Corollary $\ref{Corollary:refinementErdos}$]
We use the implicit function Theorem (for instance Theorem 2 of \cite{bur1}) to deduce that there exist two differentiable functions $\varphi$ and $\psi$, defined on nonempty open subsets $\calI'$ of $\calI$ and $\calJ'$ of $\calJ$, such that $P\bigl(x,\varphi(x)\bigr)=0$ and $P\bigl(\psi(y),y\bigr)=0$ for $x\in\calI'$ and $y\in\calJ'$, and such that $\varphi\circ\psi$ is the identity on $\calJ'$ and 
$\psi\circ\varphi$ is the identity on $\calI'$. 
We then apply Corollary $\ref{Corollaire:unseulphi}$. 
\end{proof}

Erd\H{o}s's result on $t=\xi+\eta$ for $t\in\R$ follows from Corollary $\ref{Corollary:refinementErdos}$ with $P(X,Y)=X+Y-t$, while his result on $t=\xi\eta$ for $t\in\R^\times$ follows with $P(X,Y)=XY-t$. Also, the above mentioned fact that any positive real number $t$ is the sum of two squares of Liouville numbers follows by applying Corollary $\ref{Corollary:refinementErdos}$ to the polynomial $X^2+Y^2-t$.

One could also deduce, under the hypotheses of Corollary $\ref{Corollary:refinementErdos}$, the existence of  one pair of Liouville numbers $(\xi,\eta)$ with $P(\xi,\eta)=0$ by applying Theorem 1 of \cite{bur1} with $f(x,y)=P(x,y)$ and $\alpha=0$. The proof we gave 
produces an uncountable set of solutions. 

\bigskip
We extend Corollary $\ref{Corollary:refinementErdos}$ to more than $2$ variables as follows: 

\begin{corollary}\label{Corollary:ellVariables}
Let $\ell\ge 2$ and let $P\in\R[X_1,\dots,X_\ell]$ be an irreducible polynomial such that $(\partial /\partial X_1) P \not=0$ and $(\partial /\partial X_2) P \not=0$. Assume that there exist nonempty open subsets $\calI_i$ of $\R$ ($i=1,\ldots,\ell$) such that, for any $i\in\{1,2\}$ and any $(\ell-1)$--tuple $ (x_1,\dots,x_{i-1},x_{i+1},\dots,x_\ell)
\in\calI_1\times\cdots\times\calI_{i-1}\times\calI_{i+1}\times\cdots\times\calI_\ell$, 
there exists $x_i\in\calI_i$ such that $P(x_1,\ldots,x_\ell)=0$. 
Then there exist uncountably many tuples $(\xi_1,\xi_2,\ldots,\xi_\ell)\in \calI_1\times\calI_{2}\times\cdots \times\calI_\ell$ of Liouville numbers such that $P(\xi_1,\xi_2,\ldots,\xi_\ell)=0$. 
\end{corollary}

\begin{proof}
For $\ell=2$, this is Corollary $\ref{Corollary:refinementErdos}$. Assume $\ell\ge 3$. Using  Lemma $\ref{lemma:algind}$, we know that there exists a $(\ell-2)$--tuple of $\Q$--algebraically independent Liouville numbers $(\xi_3, \dots,\xi_\ell)$ 
in $\calI_3\times\cdots\times\calI_\ell$.
{\tt We finally apply Corollary $\ref{Corollary:refinementErdos}$ to the polynomial $P(X_1,X_2,\xi_3,\ldots,\xi_\ell)\in\R[X_1,X_2]$.} 
\end{proof}

 In \cite{bur2}, using a counting argument together with an application of B\'{e}zout's Theorem, E.B.~Burger proved that {\it an irrational number $t$ is transcendental if and only if there exist two $\Q$--algebraically independent Liouville numbers $\xi$ and $\eta$ such that $t = \xi + \eta$. } Extending the method of \cite{bur2}, we prove:

\begin{proposition}\label{Proposition:algebraicallyindependentrepresentation}
Let $F(X,Y)\in\Q[X,Y]$ be a nonconstant polynomial with rational coefficients and $t$ a real number. Assume that there is an uncountable set of pairs of Liouville numbers $(\xi,\eta)$ such that $F(\xi,\eta)=t$. Then the two following conditions are equivalent. 
\\
$(i)$ $t$ is transcendental. 
\\
$(ii)$ there exist two $\Q$--algebraically independent Liouville numbers $(\xi,\eta)$ such that $F(\xi,\eta)=t$.
\end{proposition}

\begin{proof}[Proof of Proposition $\ref{Proposition:algebraicallyindependentrepresentation}$]
Assume $t$ is algebraic. Therefore there exists $P(X) \in \Q[X]\setminus\{0\}$ such that $P(t) = 0$. For any pair of Liouville numbers $(\xi,\eta)$ such that $F(\xi,\eta)=t$, we have $P(F(\xi, \eta)) = 0$. Since $P\circ F\in
\Q[X,Y]\setminus\{0\}$, we deduce that the numbers $\xi$ and $\eta$ are algebraically dependent. 

Conversely, assume that for any pair of Liouville numbers $(\xi,\eta)$ such that $F(\xi,\eta)=t$, the numbers $\xi$ and $\eta$ are algebraically dependent. Since $\Q[X,Y]$ is countable and since there is an uncountable set of such pairs of Liouville numbers $(\xi,\eta)$, there exists a nonzero polynomial $A\in\Q[X,Y]$ such that $A(X,Y)$ and $F(X,Y)-t$ have infinitely many common zeros $(\xi,\eta)$. We use B\'{e}zout's Theorem. We decompose $A(X,Y)$ into irreducible factors in $\Qbar[X,Y]$, where $\Qbar$ is the algebraic closure of $\Q$. One of these factors, say $B(X,Y)$, divides $F(X,Y)-t$ in $\overline{\Q(t)}[X,Y]$, where $\overline{\Q(t)}$ denotes the algebraic closure of $\Q(t)$. 

Assume now that $t$ is transcendental. 
{\tt Write $F(X,Y)-t=B(X,Y)C(X,Y)$ where $C\in \overline{\Q(t)} [X,Y]$. The coefficient of a monomial $X^iY^j$ in $C$ is 
$$
\left(
\frac{\partial^{i+j}}{\partial X^i \partial Y^j}
\right)
\left(
\frac{F(X,Y)-t}{B(X,Y)}\right)
(0,0),
$$
hence  
$C\in \Qbar [t,X,Y]$ and  $C$ has degree $1$ in $t$, say $C(X,Y)=D(X,Y)+t E(X,Y)$, with $D$ and $E$ in $\Qbar [t,X,Y]$.
Therefore }
$B(X,Y)E(X,Y)=-1$, contradicting the fact that $B(X, Y)$ is irreducible. 
\end{proof}

Here is a consequence of Corollary $\ref{Corollary:AS}$ where a sequence of $(f_n)$ is involved, not only one $\varphi$ like in Corollary $\ref{Corollaire:unseulphi}$.

\begin{corollary}\label{Corollary:refinementErdosbis}
Let $\calE$ be a countable subset of $\R$. Then there exists an uncountable set of positive Liouville numbers $\xi$ having simultaneously the following properties. 
\\
(i) For any $t\in \calE$, the number $ \xi+t$ is a Liouville number. 
\\
(ii) For any nonzero $t\in \calE$,  the number $ \xi t$ is a Liouville number. 
\\
(iii) Let $t \in \calE$, $t\ne 0$. Define inductively $\xi_0=\xi$ and $\xi_n=e^{t \xi_{n-1}}$ for $n\ge 1$. Then all numbers of the sequence $(\xi_n)_{n\ge 0}$ are Liouville numbers. 
\\
(iv) For any rational number $r\ne 0$, the number $\xi^r$ is a Liouville number.
 \end{corollary}

 \begin{proof}[Proof of Corollary $\ref{Corollary:refinementErdosbis}$]
 Each of the four following sets of continuous real maps defined on $\calI=(0,+\infty)$ is countable, hence their union is countable. We enumerate the elements of the union and we apply Corollary $\ref{Corollary:AS}$.
 
 The first set consists of the maps $x\mapsto x+t$ for $t\in \calE$. 
 
The second set consists of the maps $x\mapsto xt$ for $t\in \calE$, $t\ne 0$. 
 
 The third set consists of the maps $\varphi_n$ defined inductively by $\varphi_0(x)=x$, $\varphi_n(x)=e^{t\varphi_{n-1}(x)}$ ($n\ge 1$). 
 
 The fourth set consists of the maps $\varphi_r(x) = x^r$ for any rational number $r \ne 0$.
 \end{proof}

 In \cite{Maillet}, \'{E}.~Maillet gives a necessary and sufficient condition for a positive Liouville number $\xi$ to have a $p$--th root (for a given positive integer $p>1$) which is also a Liouville number: among the convergents in the continued fraction expansion of $\xi$, infinitely many should be $p$--th powers. He provides explicit examples of Liouville numbers having a $p$--th root which is not a Liouville number. 
 
  Let $\calI$ be an interval of $\R$ with nonempty interior and $\varphi:\calI\rightarrow \calI$ a continuous bijective map (hence  $\varphi$ is nowhere locally constant). Let $\psi:\calI\rightarrow\calI$ denotes the inverse bijective map of $\varphi$. For $n\in\Z$, we denote by 
$\varphi^n $ the  bijective map $\calI\rightarrow\calI$ defined inductively as usual: $\varphi^0$ is the identity, $\varphi^{n}=\varphi^{n-1}\circ\varphi$ for $n\ge 1$, and $\varphi^{-n}=\psi^n$ for $n\ge 1$.  
 
Here is a further consequence of Corollary $\ref{Corollary:AS}$. 
 
 \begin{corollary}\label{Corollary:orbitsLiouville}
 Let $\calI$ be an interval of $\R$ with nonempty interior and $\varphi:\calI\rightarrow \calI$ a continuous bijective map.  Then the set of elements $\xi$ in $\calI$ such that the orbit $\{\varphi^n(\xi)\; \mid n\in\Z\}$ consists only of Liouville numbers in $\calI$ is a $G_\delta$--subset of $\calI$, hence is uncountable.
 \end{corollary}

 \begin{proof}[Proof of corollary $\ref{Corollary:orbitsLiouville}$]
 In Proposition $\ref{Proposition:phiandpsiRefined}$, take $X = \calI$, $\calN = \Z$,  $G_n = \mathbbL\cap \calI$ and $f_n = \varphi^n$ for each $n\in \Z$. 
  
 
 \end{proof}


 \section{Algebraic independence}\label{section:ai}


Schanuel's Conjecture states that, given $\Q$--linearly independent complex numbers $x_1,\ldots,x_n$, the transcendence degree over $\Q$ of the field 
\begin{equation}\label{Equation:Schanuel}
\Q(x_1,\ldots,x_n,e^{x_1},\ldots,e^{x_n})
\end{equation}
is at least $n$. One may ask whether the transcendence degree is at least $n+1$ when the following additional assumption is made: for each $i=1,\dots,n$, one at least of the two numbers $x_i$, $e^{x_i}$ is a Liouville number. 

We will show that for each pair of integers $(n,m)$ with $n\geq m\geq 1$, there exist uncountably many tuples $\xi_1,\ldots,\xi_n$ consisting of $\Q$--linearly independent real numbers, 
such that 
the numbers $\xi_1,\ldots,\xi_n,e^{\xi_1},\ldots,e^{\xi_n}$ are all Liouville numbers, and 
the transcendence degree of the field $(\ref{Equation:Schanuel})$ is $n+m$.

\medskip 

 For $L$ a field and $K$ a subfield of $L$, we denote by $\trd_K L$ the transcendence degree of $L$ over $K$. 
 
 \begin{theorem}\label{Th:SchanuelLiouville}
 Let $n\geq 1$ and $1\le m\leq n$ be given integers. Then there exist uncountably many $n$-tuples $(\alpha_1, \dots, \alpha_n) \in \mathbbL^n$ such that $\alpha_1,\ldots,\alpha_n$ are linearly independent over $\Q$, $e^{\alpha_i} \in \mathbbL$ for all $i = 1, 2, \dots, n$ and 
$$
\trd_\Q\Q(\alpha_1, \dots, \alpha_n, e^{\alpha_1}, \dots, e^{\alpha_n}) = n+m.
$$
\end{theorem}

\noindent{\bf Remark.} Theorem $\ref{Th:SchanuelLiouville}$ is tight when $n=1$: the result does not hold for $m=0$. Indeed, since the set of $\alpha$ in $\mathbbL$ such that $\alpha$ and $e^\alpha$ are algebraically dependent over $\Q$ is countable, one cannot get uncountably many $\alpha\in\mathbbL$ such that $\trd_{\Q}\Q(\alpha, e^\alpha) =1$. 

We need an auxiliary result (Corollary $\ref{corollary:indalgexp}$). Corollary $\ref{corollary:indalgexp}$ will be deduced from the following Proposition $\ref{Proposition:algIndPoly}$. 

\begin{proposition}\label{Proposition:algIndPoly}
\null\hfill\break
$(1)$ 
 Let $g_1,g_2,\dots, g_n$ be polynomials in $\C[z]$. Then the two following conditions are equivalent.
 \\
$ (i)$ For $1\le i<j\le n$, the function $g_i-g_j$ is not constant. 
 \\
 $(ii)$ The functions $e^{g_1},\dots, e^{g_n}$ are linearly independent over $\C(z)$.
 \\
$(2)$ Let $f_1, f_2,\dots, f_m$ be polynomials in $\C[z]$. Then the two following conditions are equivalent.
 \\
$(i)$ For any nonzero tuple $(a_1,\dots, a_m)\in \Z^m$, the function $a_1f_1+\cdots+a_mf_m$ is not constant. 
 \\
$(ii)$ The functions $e^{f_1}$, $\dots$, $e^{f_m}$
 are algebraically independent over $\C(z).$
\end{proposition}

Since the functions $1,z,z^2,\dots,z^m,\dots$ are linearly independent over $\C$, we deduce from $(2)$: 

\begin{corollary}\label{corollary:indalgexp}
 The functions
 $$
 z,\; e^z,\; e^{z^2},\; \dots, \; e^{z^m},\dots
 $$
 are algebraically independent over $\C.$
\end{corollary}
\bigskip

For the proof of Proposition $\ref{Proposition:algIndPoly}$, we introduce the quotient vector space $\calV=\C[z]/\C$ and the canonical surjective linear map $s:\C[z]\rightarrow \calV$ with kernel $\C$. Assertion $(i)$ in $(1)$ means that 
$s(g_1),\dots,s(g_n)$ are pairwise distinct, while assertion $(i)$ in $(2)$ means that 
$s(f_1),\dots,s(f_m)$ are linearly independent over $\Q$. 

\begin{proof}[Proof of $(1)$]
\null\hfill\break\noindent
$(ii)$ implies $(i)$. If $g_1-g_2$ is a constant $c$, then $(e^{g_1},e^{g_2})$ is a zero of the linear form $X_1-e^cX_2$. 

\null\hfill\break\noindent
$(i)$ implies $(ii)$.
We prove this result by induction on $n$. For $n=1$, there is no condition on $g_1$, the function $e^{g_1}$ is not zero, hence the result is true. Assume $n\ge 2$ and assume that for any $n'<n$, the result holds with $n$ replaced by $n'$. Let $A_1,\ldots,A_n$ be polynomials in $\C[z]$, not all of which are zero; consider the function 
$$
G(z)=
A_1(z)e^{g_1(z)}+\cdots+ A_n(z)e^{g_n(z)}.
$$
The goal is to check that $G$ is not the zero function. Using the induction hypothesis, we may assume $A_i\not=0$ for $1\le i\le n$. Define $h_i=g_i-g_n$ $(1\le i\le n)$ and $H=e^{-g_n}G$:
$$
H(z)=
A_1(z)e^{h_1(z)}+\cdots+ A_{n-1}(z)e^{h_{n-1}(z)}+ A_n(z).
$$
From $h_i-h_j=g_i-g_j$, we deduce that $s(h_1),\dots,s(h_{n-1})$ are distinct in $\calV$. 
Write $D=d/dz$ and let $N>\deg A_n$, so that $D^NA_n=0$. Notice that for $i=1,\dots,n-1$ and for $t\ge 0$, we can write 
$$
D^t \bigl(A_i(z)e^{h_i(z)}
\bigr)
=A_{it}(z)e^{h_i(z)}
$$
where $A_{it}$ is a nonzero polynomial in $\C[z]$. By the induction hypothesis, the function 
$$
D^NH(z)=A_{1, N}(z)e^{h_1(z)}+\cdots+ A_{n-1,N}(z)e^{h_{n-1}(z)} 
$$
is not the zero function, hence $G\not=0$. 

\end{proof}

\begin{proof}[Proof of $(2)$]
\null\hfill\break\noindent
$(ii)$ implies $(i)$. If there exists
$(a_1,\dots,a_m)\in \Z^m\setminus\{(0,\dots,0)\}$ such that the function $a_1f_1+\cdots+a_mf_m$ is a constant $c$, then for the polynomial 
$$
P(X_1,\dots,X_m)=\prod_{a_i>0} X_i^{a_i} -e^c \prod_{a_i<0} X_i^{|a_i|} 
$$
we have $P(e^{f_1},\dots,e^{f_m})=0$, therefore the functions $e^{f_1},\dots,e^{f_m}$ are algebraically dependent over $\C$ (hence over $\C(z)$). 

\null\hfill\break\noindent
$(i)$ implies $(ii)$. 
Consider a nonzero polynomial 
$$
P(X_1,\dots,X_m)=
\sum_{\lambda_1=0}^{d_1}\cdots
\sum_{\lambda_m=0}^{d_m}
p_{\lambda_1,\dots,\lambda_m} (z)X_1^{\lambda_1}\cdots X_m^{\lambda_m}\in\C[z,X_1,\dots,X_m]
$$
and let $F$ be the entire function $F=P(e^{f_1},\dots,e^{f_m})$. Denote by $\{g_1,\dots,g_n\}$ the set of functions $\lambda_1f_1+\cdots+\lambda_mf_m$ with $p_{\lambda_1,\dots,\lambda_m}\not=0$. For $1\le i\le n$, set
$$
A_i(z)=p_{\lambda_1,\dots,\lambda_m} (z)\in\C[z]
$$
where $(\lambda_1,\dots,\lambda_m)$ is defined by $g_i=\lambda_1f_1+\cdots+\lambda_mf_m$,  so that
$$
F(z)=A_1(z) e^{g_1(z)}+\cdots+A_n(z) e^{g_n(z)}.
$$
The assumption $(i)$ of $(2)$ on $f_1,\dots,f_m$ implies that the functions $g_1,\dots,g_n$ satisfy the assumption $(i)$ of $(1)$, hence the function $F$ is not the zero function. 

\end{proof}

\noindent
{\bf Remark.}
We deduced $(2)$ from $(1)$. We can also deduce $(1)$ from $(2)$ as follows. Assume that $(ii)$ in $(1)$ is not true, meaning that the functions $e^{g_1},\dots,e^{g_n}$ are linearly dependent over $\C(z)$: there exist polynomials $A_1,\dots,A_n$, not all of which are zero, such that the function 
$$
G(z)=A_1(z)e^{g_1(z)}+\cdots + A_n(z)e^{g_n(z)}
$$
is the zero function. 
Consider a set $f_1,\dots,f_m$ of polynomials such that $s(f_1)$,$\dots$, $s(f_m)$ is a basis of the $\Q$--vector subspace of $\calV$ spanned by $s(g_1),\dots,s(g_n)$. Dividing if necessary all $f_j$ by a positive integer, we may assume 
$$
s(g_i)=\sum_{j=1}^m \lambda_{ij} s(f_j) \quad (1\le i\le n)
$$
with $\lambda_{ij}\in\Z$. This means that 
$$
c_i=g_i-\sum_{j=1}^m \lambda_{ij} f_j \quad (1\le i\le n)
$$
are constants. Consider the polynomial 
$$
P(X_1,\ldots,X_m)=\sum_{i=1}^n A_i(z) e^{c_i} \prod_{j=1}^m X_j^{\lambda_{ij}}.
$$
From $P(e^{f_1},\ldots,e^{f_m})=G=0$ and from $(2)$ we deduce that this polynomial is $0$, hence the monomials 
$$
\prod_{j=1}^m X_j^{\lambda_{ij}} \quad (1\le i\le n)
$$
are not pairwise distinct: there exists $i_1\not=i_2$ with 
$$
\lambda_{i_1j}=\lambda_{i_2j} \quad\hbox{for $1\le j\le m$}.
$$
Therefore $g_{i_1}-g_{i_2}=c_{i_1}-c_{i_2}$, hence $s(g_{i_1})=s(g_{i_2})$, which means that $(i)$ in $(1)$ is not true.

\begin{proof} [Proof of Theorem $\ref{Th:SchanuelLiouville}$.]
Let $n$ and $m$ be integers such that $1\leq m\leq n$. We shall prove the assertion by induction on $m\geq 1$. Assume $m = 1$. We prove the result for all $n\geq 1$.  For each nonzero polynomial $P(X_0,X_1, \dots, X_n) \in \Q[X_0, \dots, X_n]$ in $n+1$ variables with rational coefficients, define a function 
$$
f_P: \R \rightarrow \R \mbox{ by } f_P(x) = P(x, e^x, \dots, e^{x^n}).
$$
Using Corollary $\ref{corollary:indalgexp}$, we deduce that the set $Z(f_P)$ of all real zeros of $f_p$, as the real zero locus of a non-zero complex analytic map  $f_P$,  is discrete in $\R$, hence that its complement is open and dense in $\R$. From Proposition $\ref{Proposition:alnia}$ and Baire's theorem, it follows that the set 
$$
E = 
\left\{\alpha\in\mathbbL\mid \quad \hbox{$e^{\alpha^j}\in\mathbbL $ for $j=1,\dots,n$}\right\} 
\cap
\bigcap_{P\in\Q[X_0, \dots, X_{n}]\setminus\{0\} } \left(\R\backslash Z(f_P)\right)
$$
is a $G_\delta$--subset of $\R$. Therefore, by  Corollary $\ref{Corollary:uncountable} $,  $E$ is uncountable. For any $\alpha \in E$, the numbers $\alpha, e^{\alpha}, e^{\alpha^2}, \dots, e^{\alpha^n}$ are in $\mathbbL$ and are algebraically independent over $\Q$. 
Since $\alpha$ is a Liouville number, $\alpha^2, \dots, \alpha^n$ are also Liouville numbers and $\alpha, \alpha^2, \dots, \alpha^n$ are linearly independent over $\Q$. From 
$$
\trd_{\Q}\Q(\alpha, \alpha^2, \dots, \alpha^n, e^\alpha, \dots, e^{\alpha^n}) = n+1
$$
we conclude that the assertion is true for $m =1$ and for all $n\geq 1$. 

\smallskip

Assume that $1< m \leq n$. Also, suppose the assertion is true for $m-1$ and for all $n\geq m-1$. In particular, the assertion is true for $m-1$ and $n-1$. Hence, there are uncountably many $n-1$ tuples $(\alpha_1, \dots, \alpha_{n-1}) \in \mathbbL^{n-1}$ such that $\alpha_1,\dots,\alpha_{n-1}$ are linearly independent over $\Q$,
$e^{\alpha_1}, \dots, e^{\alpha_{n-1}}$ are Liouville numbers 
 and 
\begin{equation}\label{Equation:n+m-2}
\trd_{\Q}\Q(\alpha_1, \dots, \alpha_{n-1}, e^{\alpha_1}, \dots, e^{\alpha_{n-1}}) = n+m-2.
\end{equation}
Choose such an $(n-1)$-tuple $(\alpha_1, \dots, \alpha_{n-1})$. 
Consider the subset $E$ of $\R$ which consists of all $\alpha \in \R$ such that 
$\alpha, e^\alpha$  are algebraically independent over \linebreak $\Q(\alpha_1, \dots \alpha_{n-1}, e^{\alpha_1}, \dots, e^{\alpha_{n-1}})$. 
If $P(X,Y) \in \Q(\alpha_1, \dots, \alpha_{n-1}, e^{\alpha_1}, \dots, e^{\alpha_{n-1}})[X,Y]$ is a polynomial,  define an analytic function $f(z) = P(z, e^z)$ in $\C$. Since $z, e^z$ are algebraically independent functions over $\C$ (by Corollary $\ref{corollary:indalgexp}$),  if $P$ is a nonzero polynomial, $f$ is a {\tt nonzero} function. Therefore, the set of zeros of $f$  in $\C$ is countable. Since there are only countably many polynomials $P(X, Y)$ with coefficients in the field $\Q(\alpha_1, \dots \alpha_{n-1}, e^{\alpha_1}, \dots, e^{\alpha_{n-1}})$, we conclude that $\R\backslash E$ is countable.  Therefore $F = E\cap \mathbbL$ is uncountable. 
For each $\alpha\in F$, the two numbers $\alpha, e^\alpha$ are algebraically independent over $\Q(\alpha_1, \dots, \alpha_{n-1}, e^{\alpha_1}, \dots, e^{\alpha_{n-1}})$. 
From $(\ref{Equation:n+m-2})$ 
we deduce 
$$
\trd_{\Q}\Q( \alpha_1, \dots, \alpha_{n-1},\alpha, e^{\alpha_1}, \dots, e^{\alpha_{n-1}}, e^\alpha) = n+m.
$$ 
This completes the proof of Theorem $\ref{Th:SchanuelLiouville}$. 
\end{proof}

 \goodbreak

\end{document}